\documentclass[a4paper,11pt,reqno]{amsart}

\usepackage[normalem]{ulem}
\usepackage{color}
\usepackage{amsmath}
\usepackage{amssymb}
\usepackage{array}
\usepackage{amsthm}
\usepackage[hidelinks]{hyperref}
\usepackage{enumitem}

\theoremstyle{plain}
\newtheorem{thm}{Theorem}[section]
\newtheorem{prop}[thm]{Proposition}
\newtheorem{cor}[thm]{Corollary}

\theoremstyle{definition}
\newtheorem{defn}[thm]{Definition}

\theoremstyle{remark}
\newtheorem{rem}[thm]{Remark}

\newenvironment{propprime}[1]
{\renewcommand{\thethm}{\ref{#1}$'$}
	\addtocounter{thm}{-1}
	\begin{prop}}
	{\end{prop}}

\usepackage{setspace}
\newcommand{\m}{\mathfrak{m}}
\newcommand{\p}{\partial}

\newcommand{\scal}{\mathrm{scal}}
\newcommand{\ric}{\mathrm{Ric}}
\newcommand{\trace}{\mathrm{tr}}

\newcommand{\dv}{\, dV}

\def\hg{\widehat{g}}
\def\hnab{\widehat{\nabla}}

\def\frakr{\mathfrak R}

\newcounter{mnotecount}[section]

\numberwithin{equation}{section}

\newtheorem{introthm}{\bf Theorem}

\title[]{V-static metrics and the volume-renormalised mass} 

\author[S. McCormick]{Stephen McCormick}
\address{Institutionen f\"or teknikvetenskap och matematik \\
	Lule{\aa} tekniska universitet \\
	971 87 Lule\aa \\
	Sweden} 
\email{stephen.mccormick@ltu.se}

\begin{document}

\begin{abstract}
	
	V-static metrics generalise the notion of static metrics, and stem from the work of Miao and Tam \cite{MiaoTam09}, and Corvino, Eichmair and Miao \cite{CEM} on critical points of the volume functional over the space of compact manifolds with constant scalar curvature. In this article we show that these V-static metrics arise naturally in the context of asymptotically hyperbolic manifolds as critical points of the volume-renormalised mass, recently introduced by Dahl, Kr\"oncke and the author \cite{DKM}. 
	
	In particular, we show that critical points of the volume-renormalised mass over the space of constant scalar curvature asymptotically hyperbolic manifolds without boundary, or satisfying appropriate boundary conditions, are exactly V-static metrics. This is directly analogous to the relationship between critical points of the ADM mass and static metrics for asymptotically flat manifolds.

\end{abstract}
	
\maketitle

\section{Introduction}

In recent work, Dahl and Kr\"oncke and the author introduced a new geometric quantity defined for asymptotically hyperbolic manifolds, which we call the \textit{volume-renormalised mass} \cite{DKM}. It is essentially a linear combination of the renormalised volume and a surface integral at infinity closely resembling the ADM mass for asymptotically flat manifolds (see Definition \ref{defn-VR} below for the precise definition).

Among other results, in \cite{DKM} we showed that on the set of complete constant scalar curvature asymptotically hyperbolic metrics without boundary, critical points of the volume-renormalised mass correspond exactly to Einstein metrics. In this article, critical points and local extremisers of the volume-renormalised mass are further explored, and we characterise them more generally as \textit{V-static metrics} (see Definition \ref{defn-Vstat}, below). These V-static metrics should be viewed as a generalisation of static metrics in the context of initial data for general relativity.

It should be remarked that the class of asymptotically hyperbolic manifolds we consider here impose slower decay conditions than what is usually considered in the context of mathematical general relativity, as the volume-renormalised mass is defined for metrics with slower decay than the standard asymptotically hyperbolic mass \cite{ChruscielHerzlich2003,Wang2001mass}. In particular, under our decay assumptions, the standard asymptotically hyperbolic mass is not well-defined. Roughly speaking, we are interested in metrics that are asymptotic to a fixed conformally compact asymptotically hyperbolic Einstein manifold at a rate of $o(\rho^{(n-1)/2}$), where $\rho$ is the boundary defining function and $n$ is the dimension of the manifold. See Section \ref{S-Setup} for the precise definitions.

These V-static metrics that we are interested in stem from the work of Miao and Tam \cite{MiaoTam09}, and Corvino, Eichmair and Miao \cite{CEM} as critical points of the volume functional over the space of compact manifolds with constant scalar curvature. We give a definition now.

\begin{defn}\label{defn-Vstat}
A V-static metric is Riemannian metric $g$ admitting a non-trivial solution $(f,\lambda)$ to
\begin{equation}\label{eq-V-static0}
	D\scal^*_g(f)=\lambda g
\end{equation}
where $\lambda\in\mathbb R$ and the adjoint of the linearised scalar curvature operator is given by
\begin{equation}
	D\scal^*_g(f)=-\Delta_g(f)g+\nabla^2_g(f)-f\ric_g.
\end{equation}
We call such an $f$ the static potential for $g$.

In the case of asymptotically hyperbolic manifolds, which we consider here, we will further ask that a V-static potential be bounded. Analogous to static metrics on asymptotically flat manifolds, this then implies that $f$ must be asymptotic to a constant, which depends on $\lambda$ (see Corollary \ref{cor-Vstatasymp}, below).
\end{defn} 
Clearly rescaling a solution to \eqref{eq-V-static0} by a constant results in another solution to \eqref{eq-V-static0} with $\lambda$ rescaled by the same constant. For this reason we can assume without loss of generality that bounded V-static potentials are asymptotic to $1$ at infinity, which corresponds to fixing $\lambda=n-1$. Note that V-static metrics generalise the notion of static metrics, which are solutions with $\lambda=0$ (and asymptotically hyperbolic manifolds cannot have the static potential $f$ asymptotic to a constant in this case). For complete asymptotically hyperbolic manifolds without boundary, the metric is V-static if and only if it is Einstein (Proposition \ref{prop-V-Einst}, below). In this case the equivalence of V-static metrics with critical points of the volume-renormalised mass on the space of constant scalar curvature metrics simply recovers the result of Dahl, Kr\"oncke and the author \cite{DKM} equating critical points with Einstein metrics. For comparison, recall that the critical points of the ADM mass on the space of scalar flat, complete asymptotically flat manifolds without boundary are exactly Ricci flat metrics. However, it is now well-known that minimisation of the ADM mass is closely related to boundary static metrics, which in the case of a complete asymptotically flat manifold without boundary coincide with Ricci flat metrics.

We now state a simplified version of the main results of this article. For more precise statements of Theorems \ref{thm-intro1} and \ref{thm-intro2}, the reader is directed to Theorems \ref{thm-main1} and \ref{thm-main2}, respectively.
\begin{introthm}[Theorem \ref{thm-main1}]\label{thm-intro1}
	Let $(M,g)$ be an asymptotically hyperbolic manifold without boundary. The following three statements are equivalent:
	\begin{enumerate}[label=\Roman*)]
		\item $(M,g)$ is a local extrema of the volume-renormalised mass on the space of constant scalar curvature metrics,
		\item $(M,g)$ is a V-static,
		\item $(M,g)$ is Einstein.
	\end{enumerate}
\end{introthm}
The above result in fact is straightforward to establish from \cite{DKM} and the work of Corvino, Eichmair and Miao \cite{CEM} since the only V-static potential in this case is the constant function $1$. However, we prove it in such a way that emphasises the role of the V-static potential and readily generalises to the case of a manifold with boundary (Theorem \ref{thm-intro2}). In that case, when the manifold has boundary, V-static metrics are distinct from Einstein metrics. Natural boundary conditions for this problem are to fix the Bartnik boundary data, $(\p M,g_{\p M},H)$, where $g_{\p M}$ is the induced metric and $H$ is mean curvature of the boundary. We show the following.
\begin{introthm}[Theorem \ref{thm-main2}]\label{thm-intro2}
	Let $(M,g)$ be an asymptotically hyperbolic manifold with boundary. The following two statements are equivalent:
\begin{enumerate}[label=\Roman*)]
	\item $(M,g)$ is a local extrema of the volume-renormalised mass on the space of constant scalar curvature metrics with fixed Bartnik data,
	\item $(M,g)$ is a V-static.
\end{enumerate}
\end{introthm}

The general procedure we use to prove both Theorem \ref{thm-intro1} and \ref{thm-intro2} is a now somewhat standard Lagrange multipliers argument \cite{anderson2019embeddings,HJ22,HJM,HL19,mccormick2014phase,mccormick2021hilbert}, originally due to Bartnik in the study of critical points of the ADM mass \cite{Bartnik05}. The main technical aspect in carrying out this type of argument lies in proving that the linearised scalar curvature operator (augmented with a boundary map in the case of Theorem \ref{thm-intro2}) is surjective. Conveniently, this has already been addressed in the case of a manifold without boundary by Huang, Jang and Martin \cite{HJM} (and \cite{DKM} for the precise asymptotics we use). Furthermore, the inclusion of a boundary satisfying the boundary conditions we consider here has been recently established by Huang and Jang \cite{HJ22}, although they impose decay conditions that are too strong for studying the volume-renormalised mass. Thankfully, their work only requires minor modifications to be extended to the natural decay rates for the volume-renormalised mass so we do not need to repeat the full analysis here. Instead we can directly use the results to apply the Lagrange multipliers theorem.

The structure of this article is as follows. In Section \ref{S-Setup}, we recall some basic definitions and set up the function spaces used throughout. Section \ref{S-props} then motivates the relationship between V-static metrics and the volume-renormalised mass with some elementary results that follow directly from the study of V-static metrics on compact manifolds. In Section \ref{S-noBoundary}, we prove Theorem \ref{thm-main1} (Theorem \ref{thm-intro1}) and outline the Lagrange multipliers argument from which the main theorems follow. Finally, in Section \ref{S-boundarycase}, the proof of Theorem \ref{thm-main2} (Theorem \ref{thm-intro2}) is given.
\section{Preliminaries}\label{S-Setup}
Throughout this article we will use $n\geq3$ to denote the dimension of the asymptotically hyperbolic manifolds on which we work. When we speak of mean curvature, we take it to be the trace of the second fundamental form with respect to the normal $\nu$ pointing towards the asymptotic end of our manifold. We take the Laplacian to be the trace of the Hessian, which we remark differs from \cite{DKM} wherein we defined the volume-renormalised mass, but appears more common in the literature on static metrics.

We begin by defining a suitable notion of reference manifold.
\begin{defn}
	Let $(N,h)$ be a closed $(n-1)$-dimensional manifold and $k\in\{-1,0,1\}$ a constant. An asymptotically hyperbolic reference manifold will be taken to be a manifold $\widehat M=(r_k,\infty)\times N$ equipped with a metric $\hg=\frac{1}{r^2+k}dr^2+r^2h$ that is conformally compact with all sectional curvature asymptotic to $-1$ at infinity, where $r_k=0$ if $k=0$ or $k=1$ and $r_k=1$ if $k=-1$. Furthermore, we will ask that a reference manifold be asymptotically Poincar\'e Einstein (APE) in the sense of Definition \ref{defn-APE} (below).
\end{defn}
This definition follows that of \cite{HJ22} as we rely heavily on their analysis. However, a major point of difference between this and the work there -- and indeed the majority of the work related to asymptotically hyperbolic manifolds -- is the rate at which $g$ decays to $\hg$. It is standard to ask that $(g-\hg)=O(r^{-\tau})$ for some $\tau>\frac{n}{2}$, as this is required for the usual definition of the mass of an asymptotically hyperbolic manifold to be well-defined. However, the volume-renormalised mass only requires $\tau>\frac{n-1}{2}$ to be well-defined (and a weaker integrability condition for the scalar curvature) \cite{DKM}. In order to make the decay rates precise, we make use of weighted H\"older spaces $C^{k,\alpha}_{\delta}=r^{-\delta}C^{k,\alpha}$, equipped with the standard norm
\begin{equation*}
	\|u\|_{k,\alpha,\delta}=\|r^\delta u\|_{C^{k,\alpha}}.
\end{equation*}
This follows the convention that a function in $C^{k,\alpha}_\delta$ is $O(r^{-\delta})$. Weighted H\"older spaces of sections of bundles are defined analogously (see Lee \cite{Lee06}, for example). Throughout this article we fix some $\tau\in(\frac{n-1}{2},n)$, which will serve as the rate of decay of a metric $g$ towards a fixed reference metric. The lower bound suffices to ensure that the volume-renormalised mass is well-defined, and the upper bound is a requirement to ensure that $(\Delta-n)$ is an isomorphism from $C^{k,\alpha}_\tau$ to $C^{k-2,\alpha}_\tau$ \cite{lee1995spectrum,Lee06}.
\begin{defn}\label{defn-APE}
	We will say a reference manifold $(\widehat M,\hg)$ is \textit{asymptotically Poincar\'e--Einstein (APE)} if $|\ric_{\hg}+(n-1)\hg|_{\hg}\in C^{k-2,\alpha}_\tau$.
\end{defn}
It is required that the reference manifold be APE in order to ensure that the volume-renormalised mass with respect to that manifold is well-defined under the appropriate scalar curvature integrability condition. We will therefore impose this additionally throughout, on the fixed $(\widehat M,\hg)$.
\begin{defn}
A smooth connected Riemannian manifold $(M,g)$ is said to be asymptotic to the reference manifold $(\widehat M,\hg)$ if there exist compact sets $K\subset M$ and $\widehat K\subset \widehat M$ and a diffeomorphism $\varphi:\widehat M\setminus \widehat K\to M\setminus K$, such that
\begin{equation*}
	\varphi^* g - \hg \in C^{2,\alpha}_{\tau}(S^2T^*(\widehat M\setminus \widehat K)),
\end{equation*}
where $S^2T^*(\widehat M\setminus \widehat K)$ is the bundle of symmetric bilinear forms on $\widehat M\setminus \widehat K$
\end{defn}
Throughout this article we will slightly abuse notation while we work on by omitting reference to this diffeomorphism when working with $g$ and $\hg$ on the asymptotic end.

The space of asymptotically hyperbolic Riemannian metrics on $M$ that we consider throughout will be denoted by
\begin{equation}\label{eq-Rdefn}
	\mathcal R^{k,\alpha}_{\tau}=\{g | g-\hg_0 \in C^{2,\alpha}_{\tau}(S^2T^*M),g>0 \}
\end{equation}
where $\hg_0$ denotes some extension of $\hg$ to the whole manifold $M$.

\begin{defn}[\cite{DKM}]\label{defn-VR}
	The volume-renormalised mass of a Riemannian manifold $(M,g)$ that is asymptotic to a reference manifold $(\widehat M, \hg)$ is defined as
	\begin{align}\begin{split}\label{eq-VRmassdefn}
		\m_{VR,\hg}(g)=&\lim_{R\to\infty}\left( \int_{S_R}\left(\hnab^i(g_{ij})-\hnab_j(\hg^{ik}g_{ik})\right)\nu^j\,dS_{\hg}\right.\\
		&\left.+ 2(n-1)\left(\int_{B_R}dV_{g}-\int_{\widehat B_R}dV_{\hg}  \right)\right),
	\end{split}
\end{align}
where $S_R$ is a sphere of radius $R$ in $M\setminus K \cong \widehat M\setminus \widehat K$, and $B_R$ and $\widehat B_R$ are the regions bounded by $S_R$ in $M$ and $\widehat M$ respectively.
\end{defn}
It was shown in \cite{DKM} (Theorem 3.1 therein) that $\m_{VR,\hg}(g)$ is well-defined and finite for $g\in\mathcal R^{2,\alpha}_{\tau}$ provided that $\scal_g+n(n-1)\in L^1(M)$. Assuming the manifolds are conformally compact and a mild assumption on the conformal boundary, the volume-renormalised mass was also shown to be independent of the diffeomorphism used to define it (Theorem 3.18 of \cite{DKM}).

As mentioned in the Introduction, the main results here rely on a Lagrange multipliers argument. The Lagrange multiplier Theorem for Banach manifolds is a classical textbook result, however we state explicitly the version we use here for reference. A proof of which can be found in Appendix D of \cite{HL19}, for example.
\begin{thm}[Lagrange multipliers for Banach spaces]\label{thm-lagrange}
	Let $X$ and $Y$ be Banach spaces and $\frakr:X\to Y$ a $C^1$ map. Assume that for each $x\in \frakr^{-1}(0)$, the linearisation $D_x\frakr:X\to Y$ is surjective. Then if $D_x\mathcal H[v]=0$ for all $v\in \ker(D_x\frakr)$ for some $x\in \frakr^{-1}(0)$ and $C^1$ functional $\mathcal H:X\to \mathbb R$, there exists $\lambda\in Y^*$ such that $D_x\mathcal H[v]=\lambda(D_x\frakr[v])$ for all $v\in X$.
\end{thm} 

\vspace{5mm} %ADJUST

\section{Properties of V-static metrics}\label{S-props}
This section contains essentially a reinterpretation of some results due to Corvino, Eichmair and Miao \cite{CEM}, which although were obtained in the case of compact manifolds, in fact shed some light on the connection between V-static metrics and the volume-renormalised mass.

The following proposition, proven in \cite{CEM}, follows the same argument as the static case \cite{corvino2000scalar}, and importantly, is a local argument so applies also to asymptotically hyperbolic manifolds.
\begin{prop}[Proposition 2.1 of \cite{CEM} (cf. Theorem 7 (i) of \cite{MiaoTam09})]
	If $(M,g)$ admits a non-trivial weak solution $f\in H^1_{loc}$ to \eqref{eq-V-static0}, then $\scal_g$ is constant. That is, an asymptotically hyperbolic V-static manifold has scalar curvature equal to $-n(n-1)$.
\end{prop}
Still following the compact case (cf. example 1.3 of \cite{CEM}), taking the trace of \eqref{eq-V-static0} then implies that a V-static potential $f$ must satisfy
\begin{equation}
	(-\Delta_g+n)(f-1)=0.
\end{equation}
Since $(-\Delta_g+n):C^{2,\alpha}_\delta\to C^{0,\alpha}_\delta$ is an isomorphism for $\delta\in(-n,1)$ \cite{lee1995spectrum,Lee06}, we have $u\equiv 0$. That is, we have the following.
\begin{prop}\label{prop-V-Einst}
	Let $f$ satisfy \eqref{eq-V-static0} with $f-1\in C^{2,\alpha}_\delta$ for $\delta\in(-n,1)$ on some asymptotically hyperbolic manifold $(M,g)$ without boundary. Then $f\equiv1$ and $\ric_g=(1-n)g$. 
\end{prop}
Note that the converse is also true. That is, if $\ric_g=(1-n)g$ then $f\equiv1$ solves \eqref{eq-V-static0}. Note that if $M$ has an interior boundary then this argument no longer holds, and V-static is distinct from Einstein, just as with the equivalence between static asymptotically flat and Ricci flat. That is, unless we impose additional restrictions such as $f\equiv 1$ on the boundary. In this case, Dirichlet boundary conditions ensure that $(-\Delta_g+n)$ is still an isomorphism between weighted H\"older spaces \cite{lee1995spectrum}, so we also have the following.
\begin{prop}
	Let $f$ satisfy \eqref{eq-V-static0} with $f-1\in C^{2,\alpha}_\delta$ for $\delta\in(-n,1)$ on some asymptotically hyperbolic manifold $(M,g)$ with boundary. Assume further that $f\equiv 1$ on $\p M$, then $f\equiv1$ and $\ric_g=(1-n)g$.
\end{prop}
 A key result of \cite{CEM} is a local deformation result (see Theorems 1.1 and 1.2 therein), demonstrating that if a metric $g$ is not V-static then $g$ can deformed locally on an open set in such a way to make small prescribed changes to scalar curvature and volume simultaneously. We don't need the full power of the result here, and the precise statement is somewhat technical. For this reason we do not explicitly state the full theorem. However, we do state the following corollary of their main result.
\begin{cor}\label{cor-CEM}
	If $g\in \mathcal R^{4,\alpha}_{\tau}$ minimises the volume-renormalised mass with respect to $(\widehat M,\hg)$ on some asymptotically hyperbolic manifold $M$ (with or without interior boundary), then $g$ must be V-static.
\end{cor}
\begin{proof}
	Let $(M,g)$ be asymptotically hyperbolic with well-defined volume-renormalised mass with respect to $(\widehat M,\hg)$. If $g$ were not V-static then we can take some open domain $U\subset\subset M$ where $g$ is not V-static. By Theorem 1.2 of \cite{CEM} we can find a new metric $\tilde g$ satisfying the following properties:
	\begin{enumerate}
		\item $\tilde g$ exactly agrees with $g$ outside of $U$,
		\item $\scal_{\tilde g}\equiv \scal_g$ everywhere on $M$, and
		\item the volume of $U$ with respect to $\tilde g$ is strictly less than the volume of $U$ with respect to $g$.
	\end{enumerate}

It immediately follows that $\tilde g$ would have strictly smaller volume-renormalised mass than $g$, and therefore $g$ cannot minimise the volume-renormalised mass if it is not static.
\end{proof}
The above corollary is already a first demonstration of the volume-renormalised mass-minimising property of V-static metrics. It is worth remarking that the proof is somewhat more straightforward than the analogous result for the ADM mass and static asymptotically flat manifolds (cf. Theorem 8 of \cite{corvino2000scalar}), as we can decrease the volume-renormalised mass without affecting the boundary term at infinity. In this article, we aim to understand this connection analogously to Bartnik's variational approach to ADM mass minimisers \cite{Bartnik05}.

\section{The case of no boundary -- Einstein metrics are mass minimisers}\label{S-noBoundary}
As discussed, the relationship between V-static metrics and the volume-renormalised mass can be viewed analogously to the relationship between static metrics and the ADM mass. In fact, it is not only true that many results have direct analogues but also the proofs follow by essentially the same arguments too. Following Bartnik's approach to the ADM \cite{Bartnik05}, we define an analogue of the Regge--Teitelboim Hamiltonian and will apply a Lagrange multipliers argument to it. This modified Regge-Teitelboim functional is given by
\begin{equation}\label{eq-Hdefn0}
	\mathcal H(g)=\m_ {VR,\hg}(g)-\int_M f\left( \scal_g+n(n-1) \right)\,dV_g
\end{equation}
where $f$ is a function asymptotic to $1$ at infinity and we have suppressed reference to $\hg$ and $f$ in the notation. Although neither term in \eqref{eq-Hdefn0} is finite for general $g\in\mathcal R^{2,\alpha}_{\tau}$, in light of the renormalised Einstein--Hilbert action defined in \cite{DKM}, we can quickly convince ourselves that the dominant terms in each should cancel out resulting in something finite when appropriately formulated (see Theorem \ref{thm-DH1} below).

It should be noted that while the standard Regge--Teitelboim Hamiltonian generates the correct equations of motion for the Einstein equations, $\mathcal H$ defined by \eqref{eq-Hdefn0} cannot be viewed as a genuine Hamiltonian. However, $\mathcal H$ is closely related to a reduced Hamiltonian first developed by Fischer and Moncrief \cite{fischer1997hamiltonian,fischer2002hamiltonian}, as indicated in a forthcoming article of Dahl, Kr\"oncke and the author \cite{DKM_forthcoming}. Nevertheless, $\mathcal H$ still plays the role of a Lagrange function for extremising $\m_ {VR,\hg}$ subject to the constraint $\scal_g+n(n-1)=0$, and the Lagrange multiplier we find in this process corresponds to the V-static potential that the minimiser admits. In this section we consider the case where $(M,g)$ has no interior boundary, and then in the following section deal with an interior boundary separately although the argument in both cases is essentially the same. We will make use of the following densitised and normalised scalar curvature
\begin{equation*}
	\mathfrak R(g)=\left(\scal_g+n(n-1) \right)dV_g
\end{equation*}
so that we can write \eqref{eq-Hdefn0} as
\begin{equation}\label{eq-Hsimp0}
	\mathcal H(g)=\m_ {VR,\hg}(g)-\int_M f\,\mathfrak R(g).
\end{equation}
 This form of the Lagrange function emphasises the role of $f$ as the Lagrange multiplier, and $\frakr$ as the constraint map. We will be interested in the constraint set
\begin{equation*}
	\mathcal{C}_o=\frakr^{-1}(0)=\{ g\in \mathcal R^{2,\alpha}_{\tau} | \scal_g=-n(n-1)  \}
\end{equation*}
of constant scalar curvature asymptotically hyperbolic metrics, asymptotic to $(\widehat M,\hg)$. We omit reference to $\alpha$ and $\tau\in(\frac{n-1}2,n)$ for the sake of notational brevity and the subscript $o$ is used to differentiate the constraint set from when we consider a similar set for a manifold with boundary.

As mentioned in the introduction, the key ingredients for this work come from the work of Huang, Jang, and Martin \cite{HJM} who used similar arguments to prove rigidity of the standard asymptotically hyperbolic mass, and the work of Huang and Jang \cite{HJ22} carrying out an analogous analysis for the case with boundary. Unfortunately, since the standard asymptotically hyperbolic mass requires faster decay than we consider here, the metrics considered in \cite{HJM} and \cite{HJ22} decay to the reference metric at a rate $\tau>\frac n2$, which is too strong for our purposes. However, there are fortunately no obstacles preventing their arguments from applying to the case we considered here. In fact, the proofs of most results we require here follow unchanged for our decay rates.

In this section we consider the case where $(M,g)$ has no interior boundary, and then in the following section deal with an interior boundary separately although the argument in both cases is essentially the same. Crucial to applying the method of Lagrange multipliers is that the linearisation of the constraint map -- in this case $\scal_g+n(n-1)$ -- is surjective. This is where the technical difficulty lies, however in the case considered here this was established already by Huang, Jang, and Martin \cite{HJM} for asymptotically hyperbolic manifolds with the standard decay ($\tau\in(\frac n2,n)$) and Dahl, Kr\"oncke and the author \cite{DKM} for the decay rates considered here. In particular we have:
 \begin{prop}
 	For $g\in\mathcal C_o$, the map  $D\frakr_g:T_g\mathcal C_o\to C^{0,\alpha}_{\tau}\left(\Lambda^3T^*M\right)$ is surjective.
 \end{prop}
\begin{proof}
	Surjectivity of the linearised scalar curvature operator was established in the aforementioned work of Huang, Jang, and Martin \cite{HJM} assuming the standard decay for $g$, namely $\tau\in(\frac n2,n)$. However, this is one of the results mentioned above that apply in our case by the same proof given there, verbatim (the case considered here is also demonstrated in \cite{DKM}). Since $g\in\mathcal{C}_o$ and we have
	\begin{equation}\label{eq-densitiseddr}
		D_g\frakr[h]=D_g\scal[h]dV_g+\frac12\trace_g(h)\left(\scal_g+n(n-1)\right)dV_g,
	\end{equation}
	surjectivity of $D\frakr_g$ follows immediately from the surjectivity of the linearised scalar curvature operator.
\end{proof}

We will also need the following result of Huang, Jang and Martin to control the decay of a V-static potential.

\begin{thm}[Theorem 3.5 of \cite{HJM}]\label{thm-potential-inf}
	Let $g\in\mathcal R^{2,\alpha}\cap C^{\infty}_{loc}$ and $V\in C^{,\alpha}_{loc}(M\setminus K)$, for some compact set $K\in M$, satisfy
	\begin{equation*}
		D\scal_g^*[V]=Z,
	\end{equation*}
	for $Z\in C^{0,\alpha}_{-s}(M\setminus K)$, where $s>0$. Then one of the following holds:
	\begin{enumerate}
		\item there is some cone $U\subset M\setminus K$ and a constant $C>0$ such that
		\begin{equation*}
			C^{-1}|x|\leq |V(x)|\leq C|x|,\qquad \text{ for all }x\in U\text{, or}
		\end{equation*}
		\item there are constants $C>0$ and $0>d\leq1$ such that
		\begin{equation*}
			|V(x)|\leq C|x|^{-d}\qquad\text{ for all } x\in M\setminus K.
		\end{equation*}
	\end{enumerate}
\end{thm} 
Note that in \cite{HJM}, $Z$ is taken to be in a specific weighted H\"older space with decay relayed to the rate $\tau$ at which $g$ decays. However, this too is only because it is the decay required there so we state the result in the generality that the proof provides. Note that we need to require that $g$ be smooth in the above. This is because the proof relies on standard elliptic regularity theory, and for the same reason we will need it for our main results.

We have the following straightforward corollary of Theorem \ref{thm-potential-inf} that shows V-static potentials are either asymptotically constant of grow linearly on a cone analogous to static potentials on asymptotically flat manifolds \cite{miao2015static}.

\vspace{5mm} %ADJUST

\begin{cor}\label{cor-Vstatasymp}
		Let $g\in\mathcal R^{2,\alpha}\cap C^{\infty}_{loc}$ and $V\in C^{,\alpha}_{loc}(M\setminus K)$, for some compact set $K\in M$, satisfy
		\begin{equation*}
			D\scal_g^*[V]-\lambda g=Z,
		\end{equation*}
		for $Z\in C^{0,\alpha}_{-s}(M\setminus K)$, where $s>0$. Then one of the following holds:
		\begin{enumerate}
			\item there is some cone $U\subset M\setminus K$ and a constant $C>0$ such that
			\begin{equation*}
				C^{-1}|x|\leq |V(x)|\leq C|x|,\qquad \text{ for all }x\in U\text{, or}
			\end{equation*}
			\item there are constants $C>0$ and $0>d\leq1$ such that
			\begin{equation*}
				|V(x)-\frac{\lambda}{n-1}|\leq C|x|^{-d}\qquad\text{ for all } x\in M\setminus K.
			\end{equation*}
		\end{enumerate}
In particular, a V-static potential either grows linearly on a cone or is asymptotically constant with the constant given by $\lambda/(n-1)$ with $\lambda$ as in equation \eqref{eq-V-static0}.
\end{cor}
\begin{proof}
	Set $u=V-\frac{\lambda}{n-1}$ then note that we have 
	\begin{equation*}
		D\scal_g^*[u]=Z+\frac{\lambda}{n-1}\left(\ric_g+(n-1)g  \right).
	\end{equation*}
	Since $g$ is APE, The result follows directly from Theorem \ref{thm-potential-inf} applied to $u$.
\end{proof}

We are now ready to carry out the main proofs of this section, beginning by demonstrating the that the modified Regge--Teitelboim Hamiltonian is well-defined when suitably regularised.

\begin{thm}\label{thm-regham0}
	The functional $\mathcal H$ defined by \eqref{eq-Hdefn0} with $(f-1)\in C^{0,\alpha}_{\tau}(M)$ can be extended to a functional that is defined on all $\mathcal R^{2,\alpha}_{\tau}$, given by
	\begin{align}\begin{split}\label{eq-regham0}
		\mathcal H(g)=&\,\int_M \,\Biggl( \hnab^i\hnab^jg_{ij}-\widehat\Delta(\hg^{ij}g_{ij})+2(n-1)\left(\frac{\sqrt{g}}{\sqrt{\hg}}-1\right)\\
		&-f\left( \scal_g+n(n-1) \right)\frac{\sqrt{g}}{\sqrt{\hg}}\,\Biggr)\,dV_{\hg},
	\end{split}
	\end{align}
where $\frac{\sqrt{g}}{\sqrt{\hg}}$ is defined by $dV_g=\frac{\sqrt{g}}{\sqrt{\hg}}dV_{\hg}$.
\end{thm}
\begin{proof}
	The expression \eqref{eq-regham0} is exactly what one obtains by writing \eqref{eq-VRmassdefn} as a single combined integral over $M$ via the divergence theorem, so we need only demonstrate that the integral converges. To see this, first note that $(1-f)(\scal_g+n(n-1))$ is $O(r^{-2\tau})$ and therefore integrable. So it is equivalent to show that $\mathcal H$ defined by \eqref{eq-regham0} is finite when $f\equiv1$. This is essentially the renormalised Einstein--Hilbert action of \cite{DKM}, which was shown to be finite therein and we see this is well-defined for the same reason. First note that the volume form satisfies
\begin{equation}\label{eq-voltaylor}
	dV_g=dV_{\hg}+\frac12\hg^{ij}(g_{ij}-\hg_{ij})dV_{\hg}+O(r^{-2\tau})dV_{\hg},
\end{equation}
which follows from a Taylor expansion and using the fact that $|g-\hg|^2=O(r^{-2\tau})$. Then since 
\begin{align*}
	D_{\hg}\scal[g-\hg]&=\hnab^i\hnab^jg_{ij}-\widehat\Delta(\hg^{ij}g_{ij})-\ric_{\hg}^{ij}(g_{ij}-\hg_{ij})
\end{align*}
we arrive at
\begin{align*}
			\mathcal H(g)=&\,\int_M \,\Biggl( D_{\hg}\scal[g-\hg]+\ric_{\hg}^{ij}(g_{ij}-\hg_{ij})+(n-1)\hg^{ij}(g_{ij}-\hg_{ij})\\
		&-(\scal_g+n(n-1))\Biggr)\,dV_{\hg}+C
\end{align*}
where $C$ denotes a collection of finite terms obtained by integrating $O(r^{-2\tau})$. Then noting that we have
\begin{equation*}
	\scal_g=-n(n-1)+D_{\hg}\scal[g-\hg]+O(r^{-2\tau}),
\end{equation*}
and since $\ric_{\hg}=-(n-1)\hg$ we see that $\mathcal H$ is well-defined.
\end{proof}

For what follows, it will be useful to first note that a direct computation (see, for example, \cite{Bartnik05}) gives
\begin{align}\begin{split}\label{eq-DRdif}
	f &D_g\frakr[h]-h\cdot D_g\frakr^*[f]\\&=\nabla^i\left( f\left( \nabla^jh_{ij}-\nabla_i(g^{jk}h_{jk}) \right)-\left( h_{ij}\nabla^jf-g^{jk}h_{jk}\nabla_i f \right) \right)dV_g,
	\end{split}
\end{align}
where $D_g\frakr^*$ is the formal adjoint of $D_g\frakr$. For convenience we will write this as $\nabla^i(\mathfrak B_i)dV_g$, since this will result in boundary terms at infinity (and on the inner boundary in the following section). We now compute the variation of the Lagrange function.
\begin{thm}\label{thm-DH1}
	For all $g\in \mathcal R^{2,\alpha}_{\tau}$ and $(f-1)\in C^{0,\alpha}_{\tau}$
	\begin{equation} \label{eq-hamlin}
		D_g\mathcal{H}(h)=(n-1)\int_M\Bigl(\trace_g(h)dV_g-h\cdot D_g\frakr^*[f]\Bigr)
	\end{equation}
	for all $h\in T_g\mathcal R=C^{2,\alpha}_{\tau}(S^2T^*M)$.
\end{thm}
\begin{proof}
	We consider $\mathcal H$ as the limit of integrals over balls $B_R$ of radius $R$ so we can consider different terms separately. The variation of $\int_{B_R} f\,\frakr(g)$ in the direction $h$ can be computed via \eqref{eq-DRdif} as
	\begin{equation*}
		\int_{B_R} fD_g\frakr[h]=\int_{B_R} h\cdot D_g\frakr^*[f]+\int_{S_R}\mathfrak B_i \nu^idS_g,
	\end{equation*}
	where $S_R$ is the boundary of $B_R$. Since $\nabla f$ and $h$ are both $O(r^{-\tau})$, then we can write this as
	\begin{align*}
	\int_{B_R} fD_g\frakr[h]&=\int_{B_R} h\cdot D_g\frakr^*[f]+\int_{S_R}f\left( \nabla^jh_{ij}-\nabla_i(g^{jk}h_{jk}) \right) \nu^idS_g+o(1)\\
	&=\int_{B_R} h\cdot D_g\frakr^*[f]+\int_{S_R}\left( \nabla^jh_{ij}-\nabla_i(g^{jk}h_{jk}) \right) \nu^idS_{\hg}+o(1),
\end{align*}	
where we use the fact that $(f-1)$ and $dV_g-dV_{\hg}$ are $O(r^{-\tau})$. Since the difference of connections tensor for $\nabla$ and $\hnab$ is also $O(r^{-\tau})$, we can replace $\nabla$ with $\hnab$ and the integral over $S_R$ becomes $\int_{B_R}\hnab^i\hnab^jh_{ij}-\widehat\Delta(\hg^{ij}h_{ij})dV_{\hg}$, which exactly cancels the variation of the first two terms in \eqref{eq-regham0} coming from the surface integral at infinity. The only remaining term in \eqref{eq-regham0} to linearise is
\begin{equation*}
	\int_{B_R}2(n-1)\left(\frac{\sqrt{g}}{\sqrt{\hg}}-1\right)\,dV_{\hg},
\end{equation*}
which in the direction $h$ gives $(n-1)\trace_g(h)dV_g$. That is, putting everything together and taking the limit $R\to\infty$ we arrive at \eqref{eq-hamlin}.
\end{proof}

We are now prepared to prove the main result of this section. The Lagrange multiplier argument for critical points of a mass functional like this is quite standard now, stemming from Bartnik's work \cite{Bartnik05}, however we closely follow the argument of \cite{HJM} in particular, as we rely on their analysis.

\begin{thm}\label{thm-main1}
	Suppose $g\in\mathcal C_o\cap C^\infty_{loc}$, then the following three statements are equivalent:
	\begin{enumerate}[label=\Roman*)]
		\item For all $h\in T_g\mathcal C_o$, we have $D_g\m_{VR}[h]=0$,
		\item There exists $f$ with $(f-1)\in C^{2,\alpha}_{\tau}$ satisfying $D_g\frakr^*[f]=(n-1)g\,dV_g$,
		\item $g$ is Einstein with $\ric_g=-(n-1)g$.
	\end{enumerate}
\end{thm}
\begin{proof}
The equivalence between $(II)$ and $(III)$ is precisely Proposition \ref{prop-V-Einst} (and the comment directly below it). So we need only consider the equivalence between $(I)$ and $(II)$.

To this end, suppose first that $(I)$ holds. This implies $D_g\mathcal H[h]=0$ for all $h\in T_g\mathcal C_o$, so the hypotheses of Theorem \ref{thm-lagrange} are satisfied. This then gives us $\lambda\in(C^{0,\alpha}_{\tau}\left(\Lambda^3T^*M\right))^*$ satisfying
\begin{equation*}
	D_g\mathcal H[h]=\lambda(D_g\frakr[h])\qquad \text{ for all }h\in T_g\mathcal C_o,
\end{equation*}
which from Theorem \ref{thm-DH1} gives
\begin{equation}
	\lambda(D_g\frakr[h])+\int_M h\cdot D_h\frakr^*[f]=(n-1)\int_M \trace_g(h) dV_g,
\end{equation}
for all $h\in C^\infty_c$. In particular, as a distribution $\lambda$ is a weak solution to
\begin{equation}\label{eq-weaksoln1}
	L[\lambda]=(n-1)g-L[f],
\end{equation}
where $L$ is the de-densitised $D_g\frakr^*$ operator, $L[\cdot]dV_g=D_g\frakr^*[\cdot]$. Note that the right-hand side of \eqref{eq-weaksoln1} is given explicitly by
\begin{equation*}
	(n-1)g-D\scal_g^*(f)=(n-1)g-\nabla^2 f+\Delta_g(f)g+f\ric_g,
\end{equation*}
where we have used the fact that $g\in \mathcal C_o$. Tracing \eqref{eq-weaksoln1} results in the elliptic equation
\begin{equation}\label{eq-elllipt1}
	-\left(\Delta_g-n\right)\lambda=\left(\Delta_g-n\right)\left(f-1\right),
\end{equation}
which $\lambda$ satisfies in the weak sense. However, Since we assume $g\in C^{\infty}_{loc}$, we have by elliptic regularity, $\lambda\in C^{2,\alpha}_{loc}$. That is $\lambda$ is a strong solution to \eqref{eq-weaksoln1}. We next must show that $\lambda$ has the desired decay at infinity. Since the right-hand side of \eqref{eq-weaksoln1} is in $C^{0,\alpha}_{\tau}$, Theorem \ref{thm-potential-inf} implies that $\lambda$ either grows linearly on a cone at infinity or is in $C^{2,\alpha}_{-d}$ for some $d\in(0,1]$.

Assume for the sake of contradiction that $\lambda$ does grow linearly on $M\setminus K$ for some compact $K$, and then without loss of generality we take $\lambda>0$ on $M\setminus K$. Now consider a family of test functions $u_i\in C^\infty_c(M\setminus K$) converging in $C^{0,\alpha}_{\tau}$ to a non-negative function $u$ that is exactly equal to $|x|^{-\tau}$ outside of a compact set. Since $\lambda$ is continuous we have
\begin{equation*}
	\lambda(u)=\lim_{i\to\infty} \lambda(u_i)=\lim_{i\to\infty}\int_M \lambda u_i\dv_g,
\end{equation*} 
where we continue to abuse notation slightly using $\lambda$ to denote the linear functional and its representation by the $L^2$ inner product. By the monotone convergence theorem we have 
\begin{equation*}
	\lambda(u)=\int_M \lambda u\,\dv_g,
\end{equation*}
which blows up to infinity since $\tau<n$ and $\lambda\geq C|x|$, contradicting the fact that $\lambda$ is bounded. That is, we have $\lambda\in C^{2,\alpha}_{-d}$ for $d\in (0,1]$.

We next observe that the right-hand side of \eqref{eq-elllipt1} belongs to $C^{0,\alpha}_{\tau}\subset C^{0,\alpha}_{-d}$. Then since $-(\Delta_g+n):C^{2,\alpha}_{-d}\to C^{0,\alpha}_{-d}$ is an isomorphism, we have $f+\lambda-1\equiv 0$. That is, by \eqref{eq-weaksoln1}, we have
\begin{equation*}
	L[f+\lambda]=L[1]=(n-1)g,
\end{equation*}
or $g$ is a V-static metric with V-static potential identically equal to $1$. That is we have $(I)\implies(II)\iff(III)$, and it remains to prove $(II)\implies(I)$.

To this end, assume $(II)$ holds, that is a V-static potential $f$ exists. Defining $\mathcal H$ with this choice of $f$, Theorem \ref{thm-DH1} implies $D_g\mathcal H[h]=0$ for all $h$. From the expression \eqref{eq-Hsimp0} for $\mathcal H$ we immediately see that this implies $D_g\m_{VR}[h]=0$ for all $h\in T_g\mathcal C_o$. 
\end{proof}
\begin{rem}
	The equivalence $(I)\iff(III)$ was already established in \cite{DKM} (Corollary 4.4 therein) by different methods, and the inclusion of $(II)$ was the straightforward part of the above proof. However, the proof presented here illuminates the connection between the volume-renormalised mass and V-static metrics. Furthermore, when we consider the case with boundary, where $(III)$ is no longer equivalent to $(II)$, the equivalence $(I)\iff(II)$ continues to hold (Theorem \ref{thm-main2} below).
\end{rem}

\section{Bartnik boundary conditions -- V-static metrics are volume-renormalised mass minimisers}\label{S-boundarycase}
We now turn to the case where $(M,g)$ has an interior boundary. If we do not impose boundary conditions then it seems unreasonable to expect any situation where the volume-renormalised mass can be minimised, so we must choose appropriate boundary conditions. It is natural to impose that the boundary itself is fixed, that is the induced metric on the boundary should be fixed. However, it is also a natural condition that the mean curvature of the boundary be fixed in addition to this, which essentially corresponds to insisting that the scalar curvature be constant right up to and including and distributional contributions on the boundary.

We are therefore interested in the set
\begin{equation*}
	\mathcal C=\{ g\in \widetilde{\mathcal R}^{2,\alpha}_{\tau}\;|\; \frakr(g)=0, g_{\Sigma}= \gamma, H(\Sigma)=H_0 \}
\end{equation*}
where $\gamma$ is some fixed $(n-1)$-dimensional Riemannian metric on $\Sigma$, $H(\Sigma)$ is the mean curvature of $\Sigma$ and $H_0$ is a fixed function on $\Sigma$. We use the notation $\widetilde{\mathcal  R}^{2,\alpha}_{\tau}$ to denote the space $\mathcal  R^{2,\alpha}_{\tau}$ defined by \eqref{eq-Rdefn} in the case where the underlying manifold has a boundary, to distinguish it from the preceding sections.

 In order to conclude that $\mathcal C$ is a Banach manifold, which is a key requirement for the argument we use here, we need surjectivity of the map $T$ defined by
\begin{align}\begin{split}\label{eq-Tdefn}
	T(h)&=(D\scal_g(h),h_{\Sigma}, DH_g(h))\\
	T:\widetilde{\mathcal{R}}^{2,\alpha}_{\tau}\to C^{0,\alpha}_{\tau}&(S^2T^*M)\times C^{2,\alpha}(S^2T^*\Sigma)\times C^{1,\alpha}(\Sigma).
	\end{split}
\end{align}
In the preceding section, surjectivity of the linearised constraint map for the decay rates we consider here were readily available. However, as mentioned in the introduction, in the case with boundary we do not have this immediately available. Fortunately, Huang and Jang \cite{HJ22} prove surjectivity of this map for the standard decay of $\tau\in(\frac n2,n)$, and after a careful examination of their proof, it is clear that this choice of $\tau$ is not essential but rather it is a choice made because they study the usual asymptotically hyperbolic mass and this rate is required for it to be well-defined. In fact, their proof goes through verbatim except in only one proposition where the decay rate is explicitly required. Furthermore, only a superficial modification is required to extend to the values of $\tau$ we work with. Since the full proof is rather involved, and the modification is straightforward, we do not repeat their entire argument here and instead explain the minor adaptation required to cover the case we require. Like many results of this flavour, the proof hinges on a coercivity estimate for $D\scal_g^*$, which in this case is precisely where the range of $\tau$ is explicitly used. The proof is given in Sections 3 and 4 of \cite{HJ22}, and the following Proposition is the only part of it requiring modification.
\begin{prop}[Proposition 3.1 of \cite{HJ22}]\label{prop-HJ22}
	Let $(M,g)$ be asymptotically hyperbolic at a rate of $\tau\in(\frac n2,n)$. Then there exist constants $R_0, C>0$ such that for all $R>R_0$ and any $u\in C^\infty_c(M)$, we have
	\begin{equation}
		\|u\rho^{1/2}\|_{H^2(M\setminus B_R)}\leq C\|D\scal_g^*(u)\rho^{1/2}\|_{H^2(M\setminus B_R)},
	\end{equation}
	where $\rho=r^{-2\tau+n-\delta}$ for some $\delta<1$ sufficiently close to $1$, and $C$ depends on $n$ and $\delta$.
\end{prop}
We briefly explain how this proposition is proved and how that proof can be modified to extend the range of permissible values of $\tau$.

The proof in \cite{HJ22} follows by a direct computation, after several simplifications are made taking note of the fact that several error terms are small. First, $g$ is taken to be identical to the reference metric $\hg=\frac{1}{r^2+k}dr^2+r^2h$ since taking $R_0$ sufficiently large renders the difference negligible. For the same reasons, we set $k=0$ in the reference metric and introduce the operator
\begin{align*}
	Lu&=D\scal_{\hg}^*(u)-\frac{1}{n-1}\trace_{hg}(D\scal_{\hg}^*(u))\hg+(\hg+\frac{1}{n-1}\scal_{\hg}\hg-\ric_{hg})\\
	&=\nabla^2 u-u\hg.
\end{align*}
Then since $(\hg+\frac{1}{n-1}\scal_{hg}\hg-\ric_{hg})$ goes to zero at infinity, we can prove the estimate for $L$ instead of $D\scal_g$. The final reduction is to note that $\|\nabla^2(u)\rho\|_{L^2(M\setminus B_R)}$ can be controlled by the $L^2(M\setminus B_R)$ norms of $L(u)$ and $u$, so one need only prove
\begin{equation}
	\int_{M\setminus B_R}\left( u^2+|\nabla u|^2 \right)r^{a}\,dV_{\hg}\leq C	\int_{M\setminus B_R}|Lu|^2r^{a}\,dV_{\hg},
\end{equation}
where $a=-2\tau+n-\delta$. The proof goes on to consider two cases depending on whether the exponent $a$ is greater than or less than $-2$, which is equivalent to whether $\tau$ is less than or greater than $1+\frac12(n-\delta)$. In particular the case we would like to extend is when $a\geq-2$ (or $\tau\leq 1+\frac12(n-\delta)$, since $\delta$ will be chosen close to $1$). The requirement that $\tau>\frac{n}2$ in \cite{HJ22} implies also that they work with $a<-\delta$, that is $a\in[-2,-\delta)$. In order to replace the conditions $\tau>\frac{n}2$ with $\tau>\frac{n-1}2$, we therefore must carry out the argument in the case $a\in[-\delta,1-\delta)$. This can be achieved by following ``Case 1'' of the proof in \cite{HJ22}, which obtains the estimate beginning from the nonnegativity of the quantity
\begin{equation*}
	\int_{S_R}\left( \beta -\nabla_\nu(u) \right)^2r^a \, dS,
\end{equation*}
where $\beta\in\mathbb R$ is to be chosen carefully later. This choice of $\beta$ is precisely the difference made here. After some direct computations Huang and Jang obtain the inequality
\begin{align*}
	0\leq& \int_{S_R}\left( \beta -\nabla_\nu(u) \right)^2r^a \, dS\\
	\leq& \int_{M\setminus B_R}\left(  \beta(2n-\beta(a-1+n))+\beta^2+1-a\beta  \right)u^2r^a\,dV\\
	&+\int_{M\setminus B_R}\left(2\beta-(a-1+n)+\beta^2+1-a\beta\right)|\nabla u|^2 r^a\,dV\\
	&+\int_{M\setminus B_R}\left(2\beta u \trace(Lu)-2(Lu)(\nabla u,\nu)\right)r^a\,dV,
\end{align*}
which relies on the fact that $a\in[-2,2]$. The idea is to then show that for some choice of $\beta$ the quantities
\begin{align*}
	c_1(a,\beta)&= \beta(2n-\beta(a-1+n))+\beta^2+1-a\beta\\
	& \text{ and}\\
	c_2(a,\beta)&=(2\beta-(a-1+n)+\beta^2+1-a\beta
\end{align*}
are negative, to give
\begin{equation}\label{eq-mainestimatestep}
	0\leq -\varepsilon \int_{M\setminus B_R}\left( u^2+|\nabla u|^2 \right)r^a\,dV+\int_{M\setminus B_R} \left(  \beta u\trace(Lu)-(Lu)(\nabla u,\nu)  \right) r^a\, dV,
\end{equation}
which eventually gives the desired estimate via Cauchy--Schwarz. They key here is in choosing $\beta$ in such a way that ensures $c_1$ and $c_2$ are negative for all values of $a$ in the range under consideration -- in this case $a\in [-\delta,1-\delta)$ where $\delta<1$ is very close to $1$. For this, the choice $\beta=\frac{a}{2}$ used in \cite{HJ22} (in Case 1 therein) does not work, however choosing $\beta=\frac a2-1$ will suffice. We can readily see that $c_2(a,\frac a2-1)<0$ since
\begin{equation*}
	c_2(a,\beta)=(\beta-(\frac a2-1))^2-\frac{a^2}{4}-n,
\end{equation*}
however we need to work a little to show $c_1(a,\frac a2-1)<0$. First note that it can be expressed as a third order polynomial in $a$,
\begin{equation}\label{eq-thepolynom}
	p(a)=c_1(a,\frac a2-1)=-\frac14 a^3+(1-\frac{n}{4})a^2+2(n-1)a+3(1-n).
\end{equation}
For $\delta$ close to $1$ we can readily check that  $p(\delta)$ and $p(1-\delta)$ are both negative, so we can simply check that $p'(a)\neq0$ for any $a\in(-1,1+\varepsilon)$ for some small $\epsilon>0$. Differentiating \eqref{eq-thepolynom} with respect to $a$ and solving for $p'(a)=0$ we find the solutions to be
\begin{equation*}
	a_\pm=\frac13 (4-n)\pm \frac23\sqrt{(2-\frac n2)^2+3(n-1)}.
\end{equation*}
For all $n\geq3$ we find that $a_+>0$ and $a_-<-1$, from which we can conclude that for all $a\in[- \delta,1-\delta)$, for $\delta<1$ sufficiently close to $1$, $c_1(a,\frac a2-1)<0$. In particular, we have that \eqref{eq-mainestimatestep} holds, which suffices to establish the main estimate as in the proof of Proposition 3.1 of \cite{HJ22}. That is, Proposition \ref{prop-HJ22} holds for the range of value of $\tau$ required here, and have the following.
\begin{propprime}{prop-HJ22}[Cf. Proposition 3.1 of \cite{HJ22}]
	Let $(M,g)$ be asymptotically hyperbolic at a rate of $\tau\in(\frac {n-1}2,n)$. Then there exist constants $R_0, C>0$ such that for all $R>R_0$ and any $u\in C^\infty_c(M)$, we have
	\begin{equation}
		\|u\rho^{1/2}\|_{H^2(M\setminus B_R)}\leq C\|D\scal_g^*(u)\rho^{1/2}\|_{H^2(M\setminus B_R)},
	\end{equation}
	where $\rho=r^{-2\tau+n-\delta}$ for some $\delta<1$ sufficiently close to $1$, and $C$ depends on $n$ and $\delta$.
\end{propprime} 
As mentioned above, the rest of the proof that $T$, defined by \eqref{eq-Tdefn}, is surjective (Theorem 4.1 of \cite{HJ22}) goes through identically as in \cite{HJ22}. That is, following the arguments in sections 3 and 4 therein verbatim from this point onward, keeping the extended range of permissible values of $\tau$, we arrive at the following Theorem.

\begin{thm}[Cf. Theorem 4.1 of \cite{HJ22}]\label{thm-surjT}
	Let $(M,g)$ be an asymptotically hyperbolic manifold with compact inner boundary, asymptotic to $(\widehat M,\hg)$ at a rate of $\tau\in(\frac {n-1}2,n)$, then the map $T$ defined by \eqref{eq-Tdefn}, is surjective.
\end{thm}

We now carry out the Lagrange multiplier argument similar to the preceding section. For this, we again would like to use the function $\mathcal H$ defined by \eqref{eq-Hdefn0}. However, in this case when we regularise the functional via the divergence theorem, we obtain some additional boundary terms on $\Sigma$, which motivates the choice of boundary conditions we use.

\begin{thm}\label{thm-regham2}
	The functional $\mathcal H$ defined by \eqref{eq-Hdefn0} with $(f-1)\in C^{0,\alpha}_{\tau}(M)$ can be extended to a functional that is defined on all $\widetilde{\mathcal R}^{2,\alpha}_{\tau}$. Furthermore, its linearisation is given by
	\begin{align} \label{eq-hamlin3}
		D_g\mathcal{H}(h)=&\,\int_M  \Bigl( (n-1)  \trace_g(h)dV_g-h\cdot D_g\frakr^*[f]\Bigr)\\
		&+\int_{\p M}f\left( \nabla^jh_{ij}-\nabla_i(\trace_g(h)) \right)-\left( h_{ij}\nabla^jf-\trace_g(h)\nabla_if \right) \nu^idS_g.\nonumber
	\end{align}
	for all $h\in T_g\widetilde{\mathcal R}=C^{2,\alpha}_{\tau}(S^2T^*M)$, where $\nu$ is the unit normal pointed towards infinity.
\end{thm}
\begin{proof}
First note that the inclusion of a boundary does not affect the well-definedness of $\mathcal H$. That is, the proof of Theorem \ref{thm-regham0} applies identically except for the addition of some finite terms on the inner boundary which we do not both to explicitly write out.

In order to establish \eqref{eq-hamlin3}, follow Theorem \ref{thm-DH1} and again consider the difference
\begin{equation*}
		f D_g\frakr[h]-h\cdot D_g\frakr^*[f]=\nabla^i(\mathfrak B_i)dV_g,
\end{equation*}
given by \eqref{eq-DRdif}. The only term here that differs from the proof of Theorem \ref{thm-DH1} is that after the integration by parts, we gain the additional term
\begin{equation}
	\int_{\p M} \mathfrak B_i\nu^i\,dV_g,
\end{equation}
which is exactly the boundary term appearing in \eqref{eq-hamlin3}.
\end{proof}

We now aim to apply the Lagrange multipliers argument using the boundary-augmented constraint map
\begin{equation*}
	\frak T (g)=\left( \frakr(g),g_{|\p M}, H\right),
\end{equation*}
where $H$ is the mean curvature of $\p M$ with respect to $\nu$. Note that $D_g\mathfrak T$ is a densitised version of the $T$ operator defined by \eqref{eq-Tdefn}, and therefore is surjective for $g\in \mathcal C$ by Theorem \ref{thm-surjT}.

\begin{thm}\label{thm-main2}
	Suppose $g\in\mathcal C$ is $C^\infty_{loc}$, then the following two statements are equivalent:
	\begin{enumerate}[label=\Roman*)]
		\item For all $h\in T_g\mathcal C$, we have $D_g\m_{VR}[h]=0$,
		\item There exists $V$ satisfying $(V-1)\in C^{2,\alpha}_{\tau}$ and $D_g\frakr^*[V]=(n-1)g\,dV_g$.
	\end{enumerate}
\end{thm}
\begin{proof}
Just as in the case with no boundary, we note that $D_g\m_{VR}[h]=D_g\mathcal H[h]$ for all $h\in T_g\mathcal C=\ker(D_g\mathfrak T)$. If we first assume $(I)$ holds then the hypotheses of Theorem \ref{thm-lagrange} holds and we therefore have a Lagrange multiplier
\begin{equation*}
	\left(\lambda, \alpha,\beta \right)\in \left( C^{0,\alpha}_{\tau}\left(\Lambda^3T^*M\right)\right)^*\times\left( C^{2,\alpha}(S^2T^*\p M)  \right)^*\times\left( C^{1,\alpha}(\p M) \right)^*
\end{equation*}
that satisfies
\begin{equation*}
	D_g\mathcal H[h]=\lambda(D_g\frakr[h])+\alpha(h_{|\p M})+\beta(D_gH[h]).
\end{equation*}
Taking $h\in C^\infty_c(Int(M))$, we have
\begin{equation*}
	D_g\mathcal H[h]=\lambda(D_g\frakr[h])=\int_M  \Bigl( (n-1)  \trace_g(h)dV_g-h\cdot D_g\frakr^*[f]\Bigr),
\end{equation*}
just as in the proof of Theorem \ref{thm-main1}. That is, as a distribution $\lambda$ is a weak solution to
\begin{equation}\label{eq-Vstatx}
	D_g\frakr^*[\lambda]=(n-1)g\,dV_g-D_g\frakr^*[f],
\end{equation}
which as before is implies that $\lambda\in C^{2,\alpha}_{loc}$ by elliptic regularity, as it satisfies \eqref{eq-elllipt1}. Furthermore note that $\lambda$ is of $C^{2,\alpha}$ regularity up to the boundary from \eqref{eq-Vstatx} since $\ric_g$ is $C^{0,\alpha}$ up the boundary. For the same reason as the proof of Theorem \ref{thm-main1}, we get also see that $\lambda\to0$ at infinity. That is, $\lambda\in C^{2,\alpha}_{-d}(M)$ for some $d\in(0,1]$. Now we can write \eqref{eq-elllipt1} as
\begin{equation}\label{eq-elllipt2}
	-\left(\Delta_g-n\right)\lambda=F,
\end{equation}
for some $F\in C^{2,\alpha}_{\tau}(M)\subset C^{2,\alpha}_{-d}(M)$. Since $-(\Delta_g+n):C^{2,\alpha}_{-s}\to C^{0,\alpha}_{-s}$ equipped with Dirichlet boundary conditions is an isomorphism for all $s\in[d,\tau]$, we have $\lambda\in C^{2,\alpha}_{\tau}$. That is, $V=f+\lambda$ is the required V-static potential satisfying $(II)$.

The reverse implication, $(II)\implies (I)$, follows by again defining $\mathcal H$ with respect to $f=V$, the given V-static potential. Then from Theorem \ref{thm-regham2}, we have
\begin{equation}\label{eq-DH1}
	D_g\mathcal H[h]=\int_{\p M}f\left( \nabla^jh_{ij}-\nabla_i(\trace_g(h)) \right)-\left( h_{ij}\nabla^jf-\trace_g(h)\nabla_if \right) \nu^idS_g
\end{equation}
for all $h\in T_g\widetilde{\mathcal R}$.

To show $D_g\mathcal H[h]=0$, we first recall the linearisation of the mean curvature. In its most convenient form for us, it can be expressed as (see, for example, Lemma 5.1 of \cite{HJ22} or Proposition 3.1 of \cite{andersonkhuri2013})
\begin{equation*}
	-2D_gH[h]=\left(\nabla^j(h_{ij})-\nabla_i(\trace_g(h))\right)\nu^i+K_{AB}h^{AB}+\nabla^A(h_{iA}\nu^i),
\end{equation*}
where $A,B$ denotes indices for $\p M$ and $K$ is the second fundamental form of $\p M$. We then note that the remaining terms in \eqref{eq-DH1} can be expressed as
\begin{equation}
	\left( h_{ij}\nabla^jf-\trace_g(h)\nabla_if \right) \nu^i=\nu^ih_{iA}\nabla^Af-g^{AB}h_{AB}\nabla_i(f)\nu^i
\end{equation}
so after an integration by parts we have
\begin{equation}\label{eq-DH2}
	D_g\mathcal H[h]=\int_{\p M}-2fD_gH[h]-fK_{AB}h^{AB}+g^{AB}h_{AB}\nabla_i(f)\nu^idS_g.
\end{equation}
That is, for all $h\in T_g\mathcal C$ we have
\begin{equation*}
	D_g\m_{VR}[h]=D_g\mathcal H[h]=0,
\end{equation*}
completing the proof.

\end{proof}

We conclude with some remarks.

\begin{rem}\label{rem-Penrose}
	For metrics on $M=\mathbb R^3\setminus \overline{B_1(0)}$ where $\overline{B_1(0)}$ is the closed unit ball, asymptotic to the standard hyperbolic metric at a rate $\tau\in(2,3)$, Brendle and Chodosh proved a (renormalised-) volume comparison theorem \cite{BrendleChodosh2014}. Namely, letting $(M,\gamma_m)$ be the AdS--Schwarzschild metric of mass $m>0$, they proved that among all metrics $g$ with the same asymptotics and boundary metric and mean curvature matching the minimal surface boundary in the AdS--Schwarzschild manifold, the renormalised volume of $(M,g)$ is strictly larger than that of $(M,\gamma_m)$ unless $g=\gamma_m$.
	
	Under these decay conditions, the renormalised volume is exactly the volume-renormalised mass (the ADM boundary integral vanishes), so this result can be understood as a Riemannian Penrose inequality for the volume-renormalised mass under stronger decay assumptions. That is, the AdS--Schwarzschild manifold minimises the volume-renormalised mass among all metrics with the same asymptotic decay rate, and minimal surface boundary of the same area. Note that AdS--Schwarzschild metrics are V-static, for if they were not then by the same argument as Corollary \ref{cor-CEM} we would be able to perform a local perturbation that decreases the volume-renormalised mass, which would contract the result of Brendle and Chodosh.
	
	Combined with Theorem \ref{thm-main2}, this suggests that the volume-renormalised mass should satisfy a Riemannian Penrose inequality.
\end{rem}
	
\begin{rem}
	Throughout this article we have tried to emphasise the analogy with the ADM mass and static asymptotically flat metrics. It seems likely that the volume-renormalised mass and V-static potentials on asymptotically hyperbolic manifolds share many other analogous properties, which would be interesting to pursue further.
\end{rem}

\section*{Acknowledgements}
This work is partially supported by Stiftelsen G.S. Magnusons fond grant no. MG2023-0060.

\setstretch{1.08} %ADJUST

\bibliographystyle{abbrvurl}
\bibliography{Refs}

\end{document}